\documentclass[12pt]{amsart}

\usepackage{amsmath,amssymb,amsfonts,amsthm}
\usepackage{multirow}
\usepackage{amsmath,amssymb,amsfonts}
\usepackage{verbatim}

\makeatletter
\@namedef{subjclassname@2010}{%
  \textup{2010} Mathematics Subject Classification}
\makeatother

\newtheorem{theorem}{Theorem}[section]
\newtheorem{proposition}[theorem]{Proposition}
\newtheorem{notation}[theorem]{Notation}
\newtheorem{lemma}[theorem]{Lemma}

\newtheorem{conjecture}[theorem]{Conjecture}

\newtheorem{question}[theorem]{Question}
\numberwithin{equation}{section}

\newcommand{\bQ}{\mathbb{Q}}
\newcommand{\bF}{\mathbb{F}}

\makeatletter
\def\imod#1{\allowbreak\mkern3mu({\operator@font mod}\,\,#1)}
\makeatother

\title{Artin Prime Producing Polynomials}
\author{Amir Akbary and Keilan Scholten}
\date{\today}
\thanks{Research of the first author was  supported by NSERC. Research of the second author was supported by an NSERC USRA award.}



\address{Department of Mathematics and Computer Science \\
        University of Lethbridge \\
        Lethbridge, AB T1K 3M4 \\
        Canada}
\email{amir.akbary@uleth.ca}

\address{Department of Mathematics and Computer Science \\
        University of Lethbridge \\
        Lethbridge, AB T1K 3M4 \\
        Canada}
\email{Keilan.Scholten@outlook.com}

\begin{document}

\begin{abstract}
We define an Artin prime for an integer $g$ to be a prime such that $g$ is a primitive root modulo that prime. Let $g\in \mathbb{Z}\setminus\{-1\}$ and not be a perfect square. A conjecture of Artin states that the set of Artin primes for $g$ has a positive density. 
In this paper we study a generalization of this conjecture for the primes produced by a polynomial and explore its connection with the problem of finding a fixed integer $g$ and a prime producing polynomial $f(x)$ with the property that a long string of consecutive primes produced by $f(x)$ are Artin primes for $g$. 
By employing some results of Moree, we propose a general method for finding such polynomials $f(x)$ and integers $g$. We then apply this general procedure for linear, quadratic, and cubic polynomials to generate many examples of polynomials with very large Artin prime production length. More specifically, among many other examples, we exhibit linear, quadratic, and cubic (respectively) polynomials with  $6355$, $37951$, and $10011$ (respectively) 
consecutive Artin primes for certain integers $g$.
\end{abstract}

\maketitle

\section{Introduction}
\label{Introduction}
We define an Artin prime for an integer $g$ (for simplicity called an Artin prime) to be a prime $p$ with the property that $g$ is a primitive root modulo $p$. Let $g\in \mathbb{Z}\setminus\{-1\}$ and not be a perfect square. A celebrated conjecture of Artin states that the set of Artin primes for $g$ has a positive density. More generally for a fixed integer $g$ if we set
$$\delta_g(x):=\frac{\#\{p\leq x;~p~ \text{is an Artin prime for}~g\}}{\#\{p\leq x;~p~\text{prime}\}},$$
then the  conjecture predicts that $$\delta_g=\lim_{x\rightarrow \infty} \delta_g(x)$$
exists. Also the conjecture states that if $g=g_1 g_2^2$ is not a perfect power and its square-free part $g_1\not\equiv 1 \imod{4}$ then
$$\delta_g=A=\prod_{q~\text{prime}} \left(1-\frac{1}{q(q-1)}  \right)=0.373955813\ldots\approx \frac{3}{8}.$$
Moreover if $g={-1}$ or a perfect square then $\delta_g=0$ and 
in all other cases $\delta_g$ is a positive constant that depends on $g$ and also is a rational multiple of $A$. The absolute constant $A$ is called  Artin's constant. Artin's conjecture is unresolved. In 1967 Hooley \cite{Hoo} proved it conditionally under the assumption that for every square-free $d$ the Dedekind zeta function of the 
Kummerian fields $\mathbb{Q}(e^{2\pi i/d}, g^{1/d})$ satisfies the generalized Riemann hypothesis.

In this paper we consider a generalization of Artin's conjecture for the primes generated by polynomials with integer coefficients. For prime $q$ let
$$N_q(f)=\#\{n \imod{q};~f(n)\equiv 0 \imod{q}\}.$$ It is easy to show that if a polynomial $f(x)$ produces infinitely many primes for values $n\in \mathbb{Z}^+$, then the following three conditions hold:

(i) The leading coefficient of $f(x)$ is positive.

(ii) $f(x)$ is irreducible over $\mathbb{Z}$.



(iii) There is no prime $q$ such that $N_q(f)=q$.

\noindent An old conjecture due to Bouniakowsky \cite{B} states that the above three conditions are also sufficient.

\begin{conjecture}[{\bf Bouniakowsky}]
\label{Boun}
A polynomial $f(x)\in \mathbb{Z}[x]$ produces infinitely many primes if and only if (i), (ii), and (iii) hold.
\end{conjecture}
This conjecture is a special case of a far reaching conjecture of Schinzel \cite{Sh} (the so called Hypothesis H) on prime values of a finite collection of polynomials.
A well-known conjecture of Bateman and Horn provides a quantitative version of Schinzel's Hypothesis H. Here we state this conjecture in the case of a single polynomial. A polynomial $f(x)$ is called a \emph{prime producing polynomial} if it produces infinitely many primes.
From now on we assume that Conjecture \ref{Boun} holds (i.e. $f(x)\in \mathbb{Z}[x]$ is a prime producing polynomial if and only if conditions (i), (ii), and (iii) hold). Let
$$\pi_f(x)=\#\{0\leq n\leq x;~f(n)~\text{is prime}\}.$$
\begin{conjecture}[\bf Bateman-Horn]
Assume that $f(x)\in \mathbb{Z}[x]$ produces infinitely many primes.  As $x\rightarrow \infty$, $$\pi_f(x)\sim \frac{1}{{\rm deg}(f)}\prod_{q~{\rm prime}} \left(\frac{q-N_q(f)}{q-1} \right) \frac{x}{\log{x}}=C(f) \frac{x}{\log{x}}.$$
\end{conjecture}
The constant $C(f)$ is called the {\em prime producing constant} of $f(x)$.
It can be shown that the product defining $C(f)$ is convergent (see \cite[p. 364]{BH}). 

A congruence class modulo a positive integer $m$ is called \emph {$m$-allowable for $f(x)$} if $(f(r), m)$ $=1$ for any integer $r$ in that congruence class. 
Similarly we can define an $m$-non-allowable congruence class for $f(x)$. Thus $N_q(f)$ is the number of $q$-non-allowable classes for $f(x)$. Note that in each $m$-non-allowable class for $f(x)$ there are only finitely many $n$ for which $f(n)$ is prime, since any prime in such class is a prime divisor of $m$ and such primes can be taken as values of $f(x)$ only finitely many times.
Moreover, as a consequence of Bateman-Horn conjecture, it can be shown that the integers $n$ for which $f(n)$ is prime are asymptotically uniformly distributed over the $m$-allowable classes for $f(x)$ (see \cite[p. 112]{M2} for a proof). In other words if $a$ is in an $m$-allowable class for $f(x)$ and $A_m(f)$ is the total number of $m$-allowable classes for $f(x)$, then 
$$\lim_{x\rightarrow \infty} \frac{\#\{0\leq n\leq x;~f(n)~\text{is prime for}~n\equiv a \imod{m} \}}{\#\{0\leq n\leq x;~f(n)~\text{is prime}\}}=\frac{1}{A_m(f)}. $$
We postulate this as the following conjecture which plays an important role in our investigations in this paper.
\begin{conjecture}[\bf Uniform Distribution]
\label{uniform}
Assume that $f(x)\in \mathbb{Z}[x]$ produces infinitely many primes. Then for any positive integer $m$ the integers $n$ for which $f(n)$ is prime are asymptotically uniformly distributed over the $m$-allowable congruence classes for $f(x)$. 
\end{conjecture}

The following is proposed by Moree \cite[Conjecture 3, p. 119]{M2}.
\begin{conjecture}[\bf Generalized Artin's Conjecture]
{Assume that $f(x)\in \mathbb{Z}[x]$ produces infinitely many primes.
For an integer $g$ let
$$\delta_g(f,x):=\frac{\#\{0\leq n\leq x;~f(n)=p~ \text{is an Artin prime for}~g\}}{\#\{0\leq n\leq x;~f(n)~\text{is prime}\}}.$$
Then
$$\delta_g(f)=\lim_{x\rightarrow \infty} \delta_g(f, x)$$
exists.}
\end{conjecture}
Combining the above conjecture with Bateman-Horn's conjecture we have 
$$\#\{0\leq n\leq x;~f(n)=p~ \text{is an Artin prime for}~g\}= \delta_g(f) C(f) \frac{x}{\log{x}}+o\left(\frac{x}{\log{x}} \right),$$
as $x\rightarrow \infty$. The case $f(x)=x$ corresponds to the classical Artin conjecture. It would be interesting if similar to the classical case we could  develop a conjectural value for the density $\delta_g(f)$. This appears to be difficult. 
However, 
it seems possible to propose a conjectural density in certain cases.

\begin{conjecture}[\bf Density Expression]
\label{GAC}
Assume that $f(x)\in \mathbb{Z}[x]$ produces infinitely many primes. Let $g$ be a square-free integer with the property that all the primes produced by $f(x)$ (except finitely many) stay inert in $\mathbb{Q}(\sqrt{g})$. Then 
$$\delta_g(f)=\lim_{x\rightarrow \infty} \delta_g(f, x)$$
exists and is independent of $g$. Moreover,
\begin{equation}
\label{deltaf}
\delta_g(f)=\delta(f) := \prod_{{\rm prime}~q > 2}\left(1-\frac{\#\{s \imod{q} | f(s) \equiv 1 \imod{q}\}}{q\#\{s \imod{q} | f(s) \not\equiv 0 \imod{q} \}}\right).
\end{equation}
\end{conjecture}
In Section \ref{Heuristic} we give a heuristic argument in support of the above density expression. Also in Proposition \ref{lineardensity}, under the assumption of the generalized Riemann hypothesis for Dedekind zeta function of certain  number fields, we prove that the above conjecture is true for linear polynomials. The infinite product $\delta(f)$ was first proposed by Moree \cite{M2} as a good approximation for $\delta_g(f)$. We have done some experiments in order to see how well $\delta(f)$ approximates $\delta_g(f)$.
Using a variety of quadratics $f(x)$ and integers $g$ with the property that all the primes produced by $f(x)$ (except finitely many) stay inert in $\mathbb{Q}(\sqrt{g})$, we numerically estimated values for $\delta_g(f)$ and $\delta(f)$.
More precisely, we used the first 500000 primes in the infinite product defining $\delta(f)$ to find a value for $\delta(f)$. We then found the actual value of $\delta_g(f,X)$ by counting how many of the primes produced in the sequence $f(0),f(1), ... ,f(X)$ are Artin primes for $g$. We did this for three different values of $X$ (i.e. $500000$, $1000000$, and $5000000$)  and recorded the difference between our approximated value of $\delta(f)$ and $\delta_g(f, X)$. 
A sample of our experimental data for four quadratic polynomials is given in Table \ref{AccTable}{


\begin{table}
\centering

\begin{tabular}{|l|l|l|l|l|l|}
\hline
Polynomial  $f(x)$& Fixed $g$ &$\approx \delta(f)$& $\delta(f)-$ & $\delta(f)-$ & $\delta(f)-$ \\
& & &$\delta_g(f, 500000)$ &$\delta_g(f,1000000)$ & $\delta_g(f,5000000)$\\
\hline
\footnotesize $56417x^2+174208554651372$ & \footnotesize $1877$ & \footnotesize $0.9987863$ & \footnotesize $0.0002893$ & \footnotesize $0.0002678$ & \footnotesize $0.0000665$ \\ \hline
\footnotesize $x^2+9828324151968468548$ & \footnotesize $14458$ & \footnotesize $0.9989678$ & \footnotesize $0.0002039$ & \footnotesize $0.0003953$ & \footnotesize $0.0002410$ \\ \hline
\footnotesize $x^2+2x+9828324124393614405$ & \footnotesize $8458$ & \footnotesize $0.9988635$ & \footnotesize $0.0000252$ & \footnotesize $0.0003463$ & \footnotesize $0.0001775$ \\ \hline
\footnotesize $x^2+2x+9828324573822479829$ & \footnotesize $3$ & \footnotesize $0.9989856$ & \footnotesize $0.0000661$ & \footnotesize $0.0006741$ & \footnotesize $0.0004032$ \\ \hline
\end{tabular}
\hskip 2cm
\caption{ Numerical results on $\delta(f)-\delta_g(f)$ }
\label{AccTable}
\end{table}

In contrast with the classical Artin constant, which has a relatively small value ($\approx 3/8$), the values of $\delta(f)$ for four quadratic polynomials recorded in Table \ref{AccTable} are very large ($\approx 1$). The existence of such polynomials was conjectured, in a related problem, first by Griffin and later was explored by Lehmer \cite{L} and Moree \cite{M2}. We now consider this problem.

Let $f(x)$ be a prime producing polynomial with integer coefficients and $g$ be an integer. Consider the sequence $(f(n))_{n=0}^\infty$.  Let $p_{i}(g, f)$ be the $i$-th prime in this sequence which is also relatively prime to $g$. Let
$$\ell_g(f)=\min\{i\in\mathbb{N};~p_i(g,f)~\mbox{is not an Artin prime for}~g\}-1.$$ 
If the above minimum does not exist we set $\ell_g(f)=\infty$.
We call $\ell_g(f)$ the  \emph{Artin prime production length of $f(x)$ with respect to $g$}.
A natural question to ask is whether it is possible to find polynomials $f(x)$ and integers $g$ with very large Artin prime production length.
The first known attempt for finding a polynomial $f(x)$ and an integer $g$ with large  $\ell_g(f)$ was carried out by Raymond Griffin. In 1957, he proposed that the decimal expansion of $1/p$ should have period length $p-1$ for all primes of the form 
$10n^2 + 7$. This is equivalent to saying that the polynomial $10x^2+7$ with $g=10$ has infinite Artin prime production length, although with modern computers we can quickly determine that this length is only 16. The problem of finding $f(x)$ and $g$ with $\ell_g(f)=\infty$ is known as Griffin's dream.
Moree has conjectured that Griffin's dream cannot be realized for prime producing quadratic polynomials.
Lehmer \cite{L} considered this problem and showed that for $g = 326$, primes produced by the polynomial $326x^2 + 3$ are expected to be Artin primes for $326$ with a probability of $0.99337...$ (This value is corrected to $0.99323\cdots$ in \cite{M2}).
It turns out that the first $206$ primes produced by $326x^2+3$ have $326$ as a primitive root. This is remarkable keeping in mind that by the classical Artin conjecture the likelihood that $206$ primes are Artin primes for $326$ is roughly
$$\left( \frac{3}{8}\right)^{206} \approx 0.1780086686\times 10^{-87}.$$
In 2007 Moree \cite{M2} generalized the method used by Lehmer in order to find many quadratic polynomials  $f(x)$ and integers $g$ with large $\ell_g(f)$.

Note that the problem of finding $f(x)$ and $g$ with large $\ell_g(f)$ is intimately related to finding $f(x)$ and $g$ with  
a large density $\delta_g(f)$. In fact the expected value of $\ell_g(f)$ can be approximated by the sum
\begin{equation}
\label{expected}
\sum_{j=1}^{\infty} j \delta_g(f)^j (1-\delta_g(f))= \frac{\delta_g(f)}{1-\delta_g(f)}.
\end{equation}
So a value of $\delta_g(f)$ close to $1$ will result in a large value for the expected Artin prime production length $\ell_g(f)$. 

We can use the density expression \eqref{deltaf} in order to find $f(x)$ and $g$ with large $\delta_g(f)$ as long as all primes produced by $f(x)$ (except finitely many) stay inert in $\mathbb{Q}(\sqrt{g})$. In other words we should have $\left(\frac{D}{p}\right)=-1$ for all (except finitely many) primes $p=f(n)$, where $D$ is the discriminant of $\mathbb{Q}(\sqrt{g})$. It is clear that this happens, under the assumption of Conjecture \ref{uniform},
if and only if $\tau_D^-(f)=1$, where 
$$\tau_D^{-}(f) = \frac{\#\{r \imod{D} | \left(\frac{D}{f(r)}\right) = -1\}}{\#\{r \imod{D} |(f(r),D) = 1\}}.$$ Here $\left(\frac{D}{.} \right)$ denotes the Kronecker symbol.
By the quadratic reciprocity and under the assumption of Conjecture \ref{uniform} one can find expressions for $\tau_D^{-}(f)$ in terms of quadratic character sums with polynomial arguments (see Theorem \ref{Theorem 1}). The computations of these sums for general polynomials are difficult, however for linear, quadratic, and some special cubics one can find explicit expressions for $\tau_D^-(f)$. We can then use these expressions to prove the following useful result.
\begin{proposition}
Assume that $f(x)=ax^n+b$ produces infinitely many primes and that the primes produced by $f(x)$ are uniformly distributed among allowable congruence classes. Let $D$ be a fundamental discriminant.

(i) If $n=1$ and $\tau_D^-(f)=1$, then $D\mid a$.

(ii) If $n=2$ and $\tau_D^-(f)=1$, then $D\mid 24a^2b$.

(iii) If $n=3$ and $\tau_D^-(f)=1$, then $D\mid 56a$.
\end{proposition}
The above proposition plays a fundamental role in our search for integers $g$ and polynomials $f(x)$ with large $\ell_g(f)$.  Part (ii) of the above proposition for a general quadratic polynomial is proved in \cite[Proposition 3]{M2}. The proofs of parts (i) and (iii) are given in Sections \ref{The Linear Case} and \ref{The Cubic Case}. The proof in the cubic case involves a careful analysis of the character sum
\begin{align*}
\psi_{q,3}(E) &= \sum\limits_{u=1}^{q-1}\left(\frac{u^3 + E}{q}\right),
\end{align*}
and its associated Jacobsthal sum
\begin{align*}
\phi_{q,3}(E) &= \sum\limits_{u=1}^{q-1}\left(\frac{u}{q}\right)\left(\frac{u^3 + E}{q}\right),
\end{align*}
where $q\equiv 1 \imod{3}$ is prime and $u$ and $E$ are integers.
A section of this paper (Section \ref{The Cubic Case}) is devoted to the calculation of these character sums. 
Generalizations of these computations to the case of a full cubic $f(x)=ax^3+bx^2+cx+d$ seem to be difficult. This is the reason that we restricted our attention in this paper to special cubics of the form $f(x)=ax^3+b$.  

Another approach to the problems considered in this paper would be to study convenient ways for producing many Artin primes.
Our examples of prime producing polynomials with a high density of Artin primes for an integer $g$ provide a simple way of producing many  Artin primes. 
We can also do similar experiments by other functions, for example one can consider Artin primes associated to  linear recurrences.

The structure of the paper is as follows.
We will review Lehmer's results and Moree's generalization in Section \ref{A General Algorithm} and based on  the ideas in \cite{M2} we describe a general method for finding Artin prime producing polynomials of a given degree with large lengths. 
We next demonstrate this method for linear polynomials in Section \ref{The Linear Case} and in Table \ref{LinTable} provide the top five linear polynomials found in our search. 
In Section \ref{The Quadratic Case} we present our modification of the presented method in \cite{M2} for quadratic polynomials. Using our modified procedure, we find a quadratic polynomial $f(x)$ and an integer $g$ with $\ell_g(f) = 37951$. The top three quadratic polynomials $f(x)$ of negative discriminant and their corresponding integers $g$ found in our search are presented in Table \ref{QuadTable}.
In Section \ref{The Cubic Case} we prove specific results for the case of cubic polynomials and present the cubic
\begin{eqnarray*}
f(x)&=&
 {16735790906636782452200520x^3 \hspace{-1pt} + \hspace{-1pt} 41975422096126566714360524823960x^2}\\
&&+\hspace{-1pt} 35093173864667750962440687534348342360x\\
&&+\hspace{-1pt} 977977739033023039412995828230137416763737,
\end{eqnarray*}
which has $g=11045$ as a primitive root for the first  $10011$ primes produced by $f(x)$.
In Section \ref{Concluding Remarks} we finish the article with some remarks and questions for future research.


A database of results found in this research, which includes lists of linear, quadratic, and cubic polynomials with large Artin prime production lengths and experimental data regarding the value of $\delta(f)$ is available at www.cs.uleth.ca/${\sim}$akbary/APPP.



\begin{notation}
Throughout this article $p$ and $q$ denote prime numbers,
$\bar{a}$ denotes the modular inverse of $a$ with respect to a given modulus, $\mathbb{F}_p$ denotes the finite field of $p$ elements, and $\left(\frac{.}{p}\right)$ denotes the Legendre symbol.
\end{notation}

\section{A General Method}
\label{A General Algorithm}

\subsection{Lehmer's Example}

We start by reviewing Lehmer's result from \cite{L} which states that a very large proportion of primes in the form $326 n^2+3$ are Artin primes for $326$.  
The following elementary lemma provides a criterion for Artin primes.

\begin{lemma}
\label{Conditions for Primitive Roots}
Let $p\nmid 2g$. Then $p$ is not an Artin prime for $g$ if and only if there exists a prime $q$ such that $q \mid p-1$ and $g^{\frac{p-1}{q}} \equiv 1 \imod{p}$.
\end{lemma}
\begin{proof}
See \cite[Theorem 4.8]{Lev}.
\end{proof}
We consider primes of the form $p = 326n^2 + 3$. Conjecture \ref{Boun} predicts that infinitely many such primes exist. By Lemma \ref{Conditions for Primitive Roots} if $p$ is not an Artin prime for $326$ then there exists a prime $q$ such that $q \mid p-1$ and $326^{\frac{p-1}{q}} \equiv 1 \imod{p}$. 
We claim that such $q$ cannot be equal to $2$, since otherwise $326^{\frac{p-1}{2}} \equiv 1 \imod{p}$ for $p=326n^2+3$ which implies that $\left(\frac{326}{p}\right) = 1$. However by the quadratic reciprocity we have
$$\left(\frac{326}{p}\right) = \left(\frac{326}{326n^2+3}\right) = \left(\frac{326n^2+3}{163}\right) = \left(\frac{3}{163}\right) = -\left(\frac{1}{3}\right) = -1.$$
Therefore $\left(\frac{326}{p}\right) \neq 1$ and so $q\neq 2$.

Now suppose that $q>2$ and $q \mid p-1= 2(163n^2 + 1)$, which can happen only if $\left(\frac{-163}{q}\right) = 1$. Note that the total number of $q$-allowable residue classes for $326x^2+3$
is $q-\left(1+\left(\frac{-978}{q}\right)\right)$. So under the condition $\left(\frac{-163}{q}\right) = 1$ there are exactly two $q$-allowable residue classes mod $q$ out of $q-\left(1+\left(\frac{-978}{q}\right)\right)$  that contains such primes.
Thus under the assumption of Conjecture \ref{uniform} we conclude that the probability that $q \mid p-1$ is $2/\left(q-1-\left(\frac{-978}{q}\right)\right)$. On the other hand, the probability that $326^{\frac{p-1}{q}} \equiv 1 \imod{p}$ (i.e. $326$ is a $q$-th power modulo $p$) is $\frac{(p-1)/q}{p-1}=\frac{1}{q}$, since the number of $q$-th power in $\mathbb{F}_p^\times$ is $(p-1)/q$ if $q\mid p-1$. 
Therefore a good approximation for the proportion of Artin primes of the form $326n^2 + 3$ for $326$ is
$$\prod\limits_{\left(\frac{-163}{q}\right) = 1}\left(1-\frac{2}{q\left(q-1-\left(\frac{-978}{q}\right)\right)}\right) = 0.99323\ldots.$$
Note that this infinite product coincides, for $f(x)=326x^2+3$, with the expression for $\delta(f)$ given in \eqref{deltaf}.
\subsection{Moree's Generalization}

In \cite{M2} Moree generalized Lehmer's method to an integer $g$ and an arbitrary prime producing polynomial $f(x)$. 
Here we describe his generalization.
Suppose that a polynomial $f(x)$ conjecturally produces infinitely primes and has large $\delta(f)$ as given in \eqref{deltaf}. In order to replicate examples similar to Lehmer's  we need to 
look for a quadratic field $\bQ(\sqrt{g})$ of discriminant $D$ with the property that all primes of the form $f(n)$ remain inert in $\bQ(\sqrt{g})$  (i.e. $\left( \frac{D}{f(n)}\right)=-1$). 
Moree \cite{M2} has devised a method for finding such quadratic fields.
Recall that for a fundamental discriminant $D$ and a polynomial $f(x)$ we defined 
$$\tau_D^{-}(f) = \frac{\#\{r \imod{D} | \left(\frac{D}{f(r)}\right) = -1\}}{\#\{r \imod{D} |(f(r),D) = 1\}}.$$ Note that $\tau_D^{-}(f)$  is a rational number. Moreover $\tau_D^{-}(f) = 1$ implies that all the primes $p=f(n)$ in $D$-allowable classes for $f(x)$ are inert in $\bQ(\sqrt{D})$.
The following result enables us to calculate $\tau_D^{-}(f)$.

\begin{theorem}[{\bf Moree}] 
\label{Theorem 1}
Let $D$ be a fundamental discriminant. Let $f(x)$ be a polynomial that generates infinitely many primes and assume that the primes of the form $f(n)$ are uniformly distributed over the $D$-allowable residue classes for $f(x)$. Let $D_1$ be the largest odd square-free divisor of $D$ and assume that $D_1 > 1$. For $j = 1,~3,~5$, and $7$, let
$$\alpha_j = \frac{\# \{ s \imod{8} | f(s) \equiv j \imod{8} \}}{4\# \{ s \imod{2} | f(s) \equiv 1 \imod{2}\}}. $$
Then we have
\begin{equation}
\label{tau}
2\tau_D^{-}(f) = \begin{cases}
1 - a_{D_1}(f) & \text{if D is odd,} \\
1 + (\alpha_3 + \alpha_7 -\alpha_1 -\alpha_5)a_{D_1}(f) & \text{if $D \equiv 4 \imod{8}$,} \\
1 + (\alpha_3 + \alpha_5 -\alpha_1 -\alpha_7)a_{D_1}(f) & \text{if $D \equiv 8 \imod{32}$,} \\
1 + (\alpha_5 + \alpha_7 -\alpha_1 -\alpha_3)a_{D_1}(f) & \text{if $D \equiv 24 \imod{32}$,}
                   \end{cases}
\end{equation}                   
where, for odd square-free $d$, $a_d(f)$ is the multiplicative function defined by
\begin{equation}
\label{a_d}
a_d(f) = \frac{\sum_{r\imod{d}} \left(\frac{f(r)}{d} \right)}{\# \{ r\imod{d} | (f(r),d) = 1 \} }.
\end{equation}
\end{theorem}
\begin{proof}
See \cite[Theorem 1]{M2}.
\end{proof}
Using this theorem we can narrow down the search for a fundamental discriminant $D$ with $\tau_D^{-}(f)=1$.
\subsection{Heuristic on Density Expression $\delta(f)$}
\label{Heuristic}
In analogy with Artin's conjecture here we describe a heuristic argument that will lead to the density expression  \eqref{deltaf}.
The elementary conditions for Artin primes given in Lemma \ref{Conditions for Primitive Roots} have the following interpretation in terms of the splitting of primes in certain algebraic number fields. Let $\zeta_q$ denote a primitive $q$-th root of unity. Consider the Kummerian field  $L_{g,q}=\mathbb{Q}(\zeta_q, g^{1/q}$). Then $g$ is a primitive root for a prime $p\nmid 2g$ if and only if there is no prime $q$ for which $p$ splits completely in $L_{g, q}$.

Let $\mathcal{P}(f)$ be the set of primes produced by $f(x)$, $k=q_1 \cdots q_s$ be a square-free positive integer, and $L_{g, k}=L_{g, q_1}L_{g, q_2}\cdots L_{g, q_s}$ be the compositum of the fields $L_{g, q_i}$ ($1\leq i \leq s$). Let $d_k(g,f)$ be the density of primes in $\mathcal{P}(f)$ that split completely in $L_{g, k}$. If $f(x)=id(x)=x$, then by the Chebotarev density theorem we know that the density $d_k(g,id)$ exists. Let us assume that $d_k(g,f)$ exists in general. So using the splitting criteria for primitive roots and by employing an inclusion-exclusion argument we arrive at $$\delta_g(f)=\sum_{k=1}^{\infty} \mu(k) d_k(g, f),$$ where $\mu(.)$ is the M\"{o}bius function.
It can be shown that if $k=q_1\cdots  q_s$ is odd and $g$ is square-free then the fields $L_{g, q_1}, \cdots , L_{g, q_s}$ are linearly disjoint over $\mathbb{Q}$ and so $d_k(g, f)=d_{q_1}(g, f)\cdots d_{q_s}(g, f)$. In other words
$$ \sum_{\substack{{k=1}\\{2\nmid k}}}^{\infty} \mu(k) d_k(g, f)=\prod_{q>2} (1-d_q(g, f)).$$
So if we can choose a square-free integer $g$ such that all the primes produced by $f(x)$ (except finitely many) stay inert in   $L_2=\mathbb{Q}(\sqrt{g})$ (in elementary terms this means that the polynomial $x^2-g$ remains irreducible over $\mathbb{F}_p$ for all primes $p=f(n)$) then 
\begin{equation}
\label{deltagf}
\delta_g(f)=\sum_{\substack{{k=1}\\{2\nmid k}}}^{\infty} \mu(k) d_k(g, f)= \prod_{q>2} (1-d_q(g, f)).
\end{equation}

We continue by finding a conjectural explicit expression for $d_q(g, f)$.
Let $d_q^1(f)$ be the density of primes $p\in \mathcal{P}(f)$ that split completely in $\mathbb{Q}(\zeta_q)/\mathbb{Q}$, and let $d_q^2(g)$ be the density of prime ideals of $\mathbb{Q}(\zeta_q)$ that split completely in $\mathbb{Q}(\zeta_q, g^{1/q})/\mathbb{Q}(\zeta_q)$. Under the assumption that $d_q(g, f)$ and $d_q^1(f)$ exist, it would be plausible to assume that  
\begin{equation*}
\label{density}
d_q(g,f)=d_q^1(f) d_q^2(g).
\end{equation*}
From the Chebotarev density theorem we know that $d_q^2(g)=[\mathbb{Q}(\zeta_q, g^{1/q}):\mathbb{Q}(\zeta_q)]=1/q$.
It is known that an odd prime $p$ splits completely in the cyclotomic field $\mathbb{Q}(\zeta_m)$ if and only if $p\equiv 1 \imod{m}$. Also $2$ splits completely in $\mathbb{Q}(\zeta_m)$ only if $m=1$ or $2$. So $d_q^1(f)$ is the density of primes of the form $f(n)$ such that $f(n)\equiv 1 \imod{q}$.  Under the assumption of Conjecture \ref{uniform} we can conclude that
\begin{eqnarray*}
d_q^1(f)&=& \lim_{x\rightarrow \infty} \frac{\#\{0\leq n\leq x;~f(n)~\text{is prime and}~f(n)\equiv 1 \imod{q} \}}{\#\{0\leq n\leq x;~f(n)~\text{is prime}\}}\\
&=&
\frac{\#\{s \imod{q} | f(s) \equiv 1 \imod{q}\}}{\#\{s \imod{q} | f(s) \not\equiv 0 \imod{q} \}}.
\end{eqnarray*}
So
$$d_q(g,f)=d_q^1(f) d_q^2(g)=\frac{\#\{s \imod{q} | f(s) \equiv 1 \imod{q}\}}{q\#\{s \imod{q} | f(s) \not\equiv 0 \imod{q} \}}.$$ Applying the above expression for $d_q(g, f)$ in \eqref{deltagf} results in \eqref{deltaf}. 

\subsection{A General Method for finding large  $\ell_g(f)$}
\label{The General Algorithm}
We can now present a general method for finding an integer $g$ and a prime producing polynomial $f(x)$ with large $\ell_g(f)$. The density expression \eqref{deltaf} and Theorem \ref{Theorem 1} are the main tools in our search for Artin prime producing polynomials of large length.
\medskip\par
\noindent{\bf General Procedure}
\begin{enumerate}
 \item Select a prime producing polynomial $f(x)\in \mathbb{Z}[x]$ such that 
 $$\delta(f)=\prod_{q > 2}\left(1-\frac{\#\{s \imod{q} | f(s) \equiv 1 \imod{q}\}}{q\#\{s \imod{q} | f(s) \not\equiv 0 \imod{q} \}}\right)$$ is very close to $1$.
 \item Use Theorem \ref{Theorem 1} to find a fundamental discriminant $D$ such that $\tau_D^{-}(f) = 1$ and then select an integer $g$ such that $D$ is the discriminant of $\bQ(\sqrt{g})$.
 \item Determine the Artin prime production length of the polynomial $f(x)$ with respect to $g$.
\end{enumerate}

We can also use two variations once we have found a polynomial $f(x)$ and an integer $g$. Firstly we can consider $f_1(x) = f(x+d)$ which is simply a shift applied to $f(x)$ and repeat step (3) for $f_1(x)$ and $g$. Secondly we can consider $g_1=k^2g$ and  vary over $k\in \mathbb{N}$ and repeat step (3) for $f(x)$ and $g_1$.

\subsection{Analysis of the General Procedure}
First of all, we note that for a given prime producing polynomial $f(x)$ it is not always possible to find a fundamental discriminant $D$ with $\tau_D^-(f)=1$. For example by employing Theorem 2.2 and Proposition 3.1 we can show that for $f(x)=ax+1$ where $a$ is a product of distinct primes in the form  $q\equiv 1$ (mod $4$), there is no fundamental discriminant $D$ with  $\tau_D^-(f)=1$. Another such example is $f(x)=x^2+x+41$ (see [12, Remark 2, p. 119]). So from a theoretical point of view the success of the above procedure depends on step (2).  

Upon finding a fundamental discriminant $D$ with the property $\tau_D^-(f)=1$, steps (1) and (2) produce a polynomial $f(x)$ and an integer $g$ with $\delta_g(f)\approx \delta(f)$ very close to $1$  (note that for square-free $g$ we expect $\delta_g(f)=\delta(f)$). Since the expected value of $\ell_g(f)$ is $\delta_g(f)/(1-\delta_g(f))$ (see (1.2))
by choosing $\delta(f)$ close to $1$ we expect that $\ell_g(f)$ will be large. 

We emphasize that a successful implementation of the above procedure will also require a moderate size for the leading coefficient (and more generally for the coefficients) of the polynomial $f(x)$ given in step (1).  It is easy to find polynomials $f(x)$ of degree $n$ that conjecturally produce infinitely many primes with corresponding $\delta(f)$ arbitrarily close to $1$. For example one can consider a polynomial $f_y(x)=a x^n +(a+2)$ with $a=q_1q_2\cdots q_m$, where $q_1, q_2, \cdots, q_m$ are all the odd primes not exceeding $y$. From the definition of $\delta(f)$ in step (1) it is evident that we can make  $\delta({f_y})$ arbitrarily close to $1$ as long as we choose $y$ large enough. However as $y\rightarrow \infty$ the leading coefficient of $f_y(x)$ grows significantly and therefore, even if we can find a suitable $D$ in step (2), the very large size of the primes produced by $f_y(x)$ will make step (3) of the procedure computationally infeasible.  

In conclusion, following our general method, the challenge in the search for pairs $(f, g)$ with large $\ell_g(f)$ is twofold. On one hand we should be able to generate prime producing polynomials $f(x)$ with large $\delta(f)$ such that their coefficients are not significantly large, on the other hand we need to devise ways to efficiently decide on the existence of the fundamental discriminants $D$ with the property $\tau_D^{-}(f)=1$ and also be able to generate such $D$'s.

In the next three sections we surmount some of these difficulties for linear, quadratic, and some cubic polynomials, by calculating the exact expressions for $\delta(f)$ (see (3.1), (4.1), (5.1)) and computing some concrete character sums in these cases. Similar calculations for polynomials of higher degrees appear to be difficult. More specifically, 
Propositions 3.1, 4.1, and 5.6 show that for linear, quadratic, and certain cubic polynomials $f(x)$ there are only finitely many potential options for a fundamental discriminant $D$ with  $\tau_D^-(f)=1$. Using these criteria one can easily generate many examples of pairs $f$ and $D$ with 
$\tau_D^-(f)=1$. Consequently, following our general method, we provide more concrete procedures for linear, quadratic, and cubic polynomials in order to produce many pairs $(f, g)$ with large $\ell_g(f)$ and report some of the examples we obtained. Our most impressive findings are for quadratic polynomials. This is partly due to the fact that the expression (4.1) for $\delta(f)$ together with the known examples of the quadratic fields with the property that a long string of consecutive primes remain inert in them, allow us to find prime producing quadratic polynomials $f(x)$ with relatively small coefficients and $\delta(f)$ very close to $1$. In contrast, in the linear case maximizing the value of (3.1) forces us to consider polynomials $f_{y,b} (x)=a x+b$, with $a=q_1q_2\cdots q_m$, where $q_1, q_2, \cdots, q_m$ are all the odd primes not exceeding $y$. Because of the large size of $a$ (as $y\rightarrow\infty$)  in the linear examples, our findings  in the linear case are modest compared to the quadratic case (our top linear example has length $6355$ while our top quadratic example has length $37951$). In the cubic case the expression (5.1) allows us to consider cubic polynomials with smaller leading coefficients (we can assume that the prime factors of the leading coefficients are not congruent to $1$ mod $3$) and therefore we can find examples of cubics with the Artin prime production length almost $1.5$ times larger than the length of our findings in the linear case (our top cubic example has length 10011).
       
\section{The Linear Case}
\label{The Linear Case}

We demonstrate our general procedure by applying it to linear polynomials.
Let $f(x) = ax + b$, where $(a, b)=1$.  By solving the corresponding congruences in \eqref{deltaf}, we find
\begin{equation}
\label{lineardelta}
\delta(f) = \prod\limits_{\stackrel{q > 2}{q \mid (a,b-1)}}\left( 1 - \frac{1}{q} \right) \prod\limits_{\stackrel{q > 2}{q \nmid a}} 
\left( 1 - \frac{1}{q(q-1)} \right).
\end{equation}
Let $q>2$. We can easily establish the following character sum identity.
$$\sum\limits_{m = 0}^{q-1} \left( \frac{am+b}{q} \right) = \begin{cases}
 q\left(\frac{b}{q}\right) & \text{if $q \mid a$}, \\
 0 & \text{if $q\nmid a$}.
                                                            \end{cases}$$
Using this sum we evaluate  \eqref{a_d} for odd primes $q$ and for $f(x)=ax+b$ to deduce that
$$a_q(f) = \begin{cases}
            \left(\frac{b}{q}\right) & \text{if $q \mid a$,} \\
            0 & \text{if $q \nmid a$.}
           \end{cases}
$$
Note that $a_d(f)$ is multiplicative on odd square free integers $d$, so if we let $D$ be a fundamental discriminant and $D_1>1$ be the largest odd square-free divisor of $D$, we get

$$a_{D_1}(f) = \begin{cases}
                \left(\frac{b}{D_1}\right) & \text{if $D_1 \mid a$,} \\
                0 & \text{if $D_1 \nmid a$.}
               \end{cases}
$$
The following simple criterion reduces the search for a fundamental discriminant $D$ with $\tau_D^{-}(f)=1$ to a finite number of steps.
\begin{proposition}
\label{linearcondition}
If $\tau_D^{-}(f)=1$, then $D\mid a$.
\end{proposition}
\begin{proof}
From Theorem \ref{Theorem 1} and the above formula for $a_{D_1}(f)$ we deduce that if $\tau_D^{-}(f)=1$, then $a_{D_1}(f)$ cannot be equal to $0$ for such $D_1$, and so $D_1 \mid a$. Now if $D$ is odd we are done. Otherwise either $D=4D_1$ and $D_1=4k+3$ or $D=8D_1$ and $D_1=2k+1$.

Let $D=4D_1$ and $D_1=4k+3$. Then from \eqref{tau} and the fact that $a_{D_1}(f)=\pm 1$, we conclude that if $\tau_D^{-}(f)=1$, then either $\alpha_1=\alpha_5=0$ or $\alpha_3=\alpha_7=0$. We assume that $\alpha_1=0$, proofs for other cases are similar. Since $\alpha_1=0$ then $an+b\equiv 1 ~({\rm mod}~{8})$ does not have any solutions and so $2\mid a$. Now since $a$ is even and $(a, b)=1$ we deduce that $b$ is odd and therefore $1-b$ is even. Since $an+b\equiv ~1 ~({\rm mod}~{8})$ does not have any solutions
we have that $4\mid a$ which together with $D_1\mid a$ imply $D=4D_1\mid a$.

Next suppose that $D=8D_1$ and $D_1=2k+1$. Now if $\tau_D^{-}(f)=1$, from \eqref{tau} we deduce that one of $\alpha_1=\alpha_7=0$, or $\alpha_3=\alpha_5=0$,  or $\alpha_5=\alpha_7=0$, or $\alpha_1=\alpha_3=0$ hold. We assume that $\alpha_1=\alpha_7=0$, proofs for other cases are similar. Since $\alpha_1=0$ and $\alpha_7=0$ then a similar reasoning as above implies that $4\mid a$. Now suppose that $b=4k+1$, then since  $an+b\equiv 1 ~({\rm mod}~{8})$ does not have any solutions we conclude that $8\mid a$. Similarly if $b=4k+3$, then since  $an+b\equiv 7 ~({\rm mod}~{8})$ does not have any solutions we again conclude that $8\mid a$.
So $8\mid a$ which together with $D_1\mid a$ imply $D=8D_1\mid a$.
\end{proof}
We next employ the above proposition together with some results from \cite{M3} and \cite{M4} to show that under certain assumptions Conjecture \ref{GAC} holds for linear polynomials.
\begin{proposition}
\label{lineardensity}
Let $f(x)=ax+b$, $a>0$, and $(a, b)=1$.  Let $g$ be a square-free integer with the property that all the primes produced by $f(x)$ (except finitely many) stay inert in $\mathbb{Q}(\sqrt{g})$. Then, assuming the generalized Riemann hypothesis for Dedekind zeta function of number fields $\mathbb{Q}(e^{2\pi i/a}, e^{2\pi i/d}, g^{1/d})$ with $d$ square-free, we have
$$\#\{0\leq n\leq x;~an+b=p~ \text{is an Artin prime for}~g\}= \frac{a\delta(f)}{\varphi(a)}\frac{x}{\log{x}}+o\left(\frac{x}{\log{x}} \right),$$
where $\varphi (.)$ is the Euler function and 
$$\delta(f)=\prod\limits_{\stackrel{q > 2}{q \mid (a,b-1)}}\left( 1 - \frac{1}{q} \right) \prod\limits_{\stackrel{q > 2}{q \nmid a}} 
\left( 1 - \frac{1}{q(q-1)} \right).$$
\end{proposition}
\begin{proof}
Under our assumptions by \cite[Theorem 2]{M3} the density $\delta_g(f)$ of Artin primes produced by $f(x)$ for $g$ exists. Moreover since all the primes produced by $f(x)$ (except finitely many) stay inert in $\mathbb{Q}(\sqrt{g})$ we conclude that $\tau_D^{-}(f)=1$, where $D$ is the discriminant of $\mathbb{Q}(\sqrt{g})$, and so by Proposition \ref{linearcondition} we have $D\mid a$. Therefore $g=4k+1$, or $g=4k+2$ and $8\mid a$, or $g=4k+3$ and $4\mid a$. For all these cases from \cite[Theorem 3]{M4} we have
\begin{equation}
\label{above}
\delta_g(f) = \prod\limits_{{q \mid (a,b-1)}}\left( 1 - \frac{1}{q} \right) \prod\limits_{{q \nmid a}} 
\left( 1 - \frac{1}{q(q-1)} \right) \left( 1-\left(\frac{g}{b}\right) \right).
\end{equation}
We observe that in all the cases $\left(\frac{g}{b}\right)=-1$. For example if $g=4k+3$ and $4\mid a$ then
for a prime in the form $an+b$ we have $$-1=\left( \frac{D}{an+b}\right) =\left( \frac{D}{Da_1n+b}\right)= \left( \frac{D}{b}\right)=\left( \frac{g}{b}\right).$$
So in \eqref{above} we have $\left(\frac{g}{b}\right)=-1$ and thus $\delta_g(f)=\delta(f)$.
\end{proof}

We now give a procedure for finding linear Artin prime producing polynomials of large length.
\medskip\par
\noindent{\bf Linear Procedure}
\begin{enumerate}
 \item Select an integer $a>0$ which is the product of many small primes. 
\item Select an integer $B$ such that $(a, B+1)=1$ and moreover either $(a,B)=1$ or $a$ and $B$ only have very large common prime divisors. 
\item Form $f(n) = an + b$ where $b = B+1$.
\item Search through divisors $D$ of $a$ that are fundamental discriminants and find ones satisfying $\tau_D^{-}(f) = 1$, then select $g$ such that $\bQ(\sqrt{g})$ has the fundamental discriminant $D$.
\item Compute $\ell_g(f)$.
\end{enumerate}

Note that the conditions on $a$ and $b$ in steps (1) and (2) ensure that $\delta(f)$ given in \eqref{lineardelta} is large, and step (4)  guarantees that the hypotheses of Conjecture \ref{GAC} hold, so that  $\delta(f)=\delta_g(f)$. So it is very likely that $\ell_g(f)$ is very large.

We have implemented the above procedure and produced many examples of Artin prime producing linear polynomials of large length. In particular we found a linear polynomial that has $1008420$ as a primitive root for the first $6355$ primes produced by the polynomial. 
We present a sample of our results in Table \ref{LinTable}. {
In the second example in Table \ref{LinTable} we have $a=\prod_{3\leq q\leq 127} q$  and in all others $a=\prod_{3\leq q\leq 101} q$.

\begin{table}
\centering

\begin{tabular}{|l|l|l|l|}
\hline
$f(x)$ &  $g$ &   $\ell_g(f)$&  $\delta(f)$ \\ \hline
\scriptsize $116431182179248680450031658440253681535x+$ & \scriptsize $1008420$ & \scriptsize $6355$ & \scriptsize $0.998271$\\
\scriptsize $33158669235192590202725416070516726471730038$ & & &\\ \hline

\scriptsize $2007238469666518094547220599513022568322942623865x+$ & \scriptsize $5$ & \scriptsize $6205$ & \scriptsize $0.998680$\\
\scriptsize $11969745093777650688032495128870012351454520519469655612$ & & &\\ \hline

\scriptsize $116431182179248680450031658440253681535x+$ & \scriptsize $9680$ & \scriptsize $5872$ &\scriptsize $0.998271$\\
\scriptsize $24446589597448128439371347196066304497192128$ & & &\\ \hline


\scriptsize $116431182179248680450031658440253681535x+$ & \scriptsize $19773$ & \scriptsize $5788$ &\scriptsize $0.998271$ \\
\scriptsize $28924300001697674192118664716361580581665158$ & & &\\ \hline

\scriptsize $116431182179248680450031658440253681535x+$ & \scriptsize $887040$ & \scriptsize $5749$ &\scriptsize $0.998271$ \\
\scriptsize $44080845573063550418381985885480043829165318$ & & &\\ \hline
\end{tabular}
\hskip2cm
\caption{Linear Polynomials With Long Artin Prime Production Lengths}
\label{LinTable}
\end{table}

\section{The Quadratic Case}
\label{The Quadratic Case}

In \cite{M2} a method for generating integers $g$ and quadratic polynomials $f(x)$ with large $\ell_g(f)$ is presented and reported that Yves Gallot by implementing that method has found the quadratic $f(x)=54151x^2 + 160744427648 x+ 119471867164612830$ of negative discriminant and $g=17431902$ with $\ell_g(f)=31082$.
Here we give a modification of the method presented in \cite{M2} to include a different range of quadratic polynomials.  We then employ our modified method to find quadratics with large Artin prime production lengths.  
The three quadratic $f(x)$ of negative discriminant and integers $g$ with $\ell_g(f)>31082$ found in our search are presented in Table \ref{QuadTable}.
We will make use of the following result.
\begin{proposition}
\label{quadratic condition}
{\bf (Moree)} Assume that $f(x) = ax^2 + bx + c$ produces infinitely many primes and that the primes produced by $f(x)$ are uniformly distributed among allowable congruence classes. Let $d=b^2-4ac$ be the discriminant of $f(x)$. 
 Let $a_{D_1}(f)$ be as defined in Theorem \ref{Theorem 1}. Then
 $$a_{D_1}(f)=\left( \frac{c}{(D_1, a, d)} \right) \left( \frac{a}{D_1/(D_1, a, d)}  \right) \prod_{\stackrel{q\mid D_1}{q\nmid ad}} \frac{-1}{q-1-\left( \frac{d}{q}\right)}.$$
Moreover if $\tau_D^{-}(f) = 1$, then $D\mid 24ad$. (Note that the first two factors in the formula for $a_{D_1}(f)$ are Kronecker symbols.) 
\end{proposition}
\begin{proof}
See \cite[Propositions 2 and 3]{M2}.
\end{proof}

Our procedure for quadratic polynomials is the following.
\medskip\par
\noindent{\bf Quadratic Procedure}

\begin{enumerate}
 \item Select an integer $\Delta$ where $\left(\frac{\Delta}{q}\right)=-1$ for many  consecutive primes $q\geq 3$.  
 \item Select an even $b$ and express it as $b = 2b'$.
 \item Factor $-\Delta + (b')^2$ into $a(c-1)$ such that $a>0$, $(a, b, c)=1$, $2\nmid (a+b, c)$, $b^2-4ac$ is not square, and moreover $a$ does not have small odd prime factors.
 \item Form $f(n) = an^2 + bn + c$.
 \item Search through divisors $D$ of $24a(b^2-4ac)$ that are fundamental discriminants and find ones satisfying $\tau_D^{-}(f)=1$, then select $g$ such that $\mathbb{Q}(\sqrt{g})$ has the fundamental discriminant $D$. 
\item Compute $\ell_g(f)$.
\end{enumerate}

We briefly explain why this works. 
By employing \eqref{deltaf} for the polynomial constructed with our method we obtain
\begin{equation}
\label{Quadratic Delta 2}
\delta(f) =  \prod\limits_{\stackrel{q \mid a}{q\mid (b, c-1)}} \left(1 - \frac{1}{q}\right)\prod\limits_{\stackrel{q\mid a}{ q\nmid b}} \left(1 - \frac{1}{q(q-1)}\right) \prod\limits_{{q \nmid a}}\left(1 - \frac{1 + \left( \frac{4(b')^2 - 4a(c-1)}{q}\right)}{q\left(q-1-\left(\frac{4(b')^2 - 4ac}{q}\right)\right)}\right),
\end{equation}
where $q$ ranges over the odd primes. Since $\left(\frac{4(b')^2 - 4a(c-1)}{q}\right) = \left(\frac{\Delta}{q}\right)=-1$ 
for many small $q$ and $a$ does not have small prime factors the value of $\delta(f)$ will be close to $1$. Therefore this polynomial will produce a high proportion of Artin primes for $g$ given in step (5) of the procedure.  

Finding $\Delta$ is a crucial step in this method. Such $\Delta$ is related to finding quadratics with a large prime producing constant. Algorithms for finding many examples of such  $\Delta$ can be found in \cite{JW}.


Our method is similar to the one originally presented in \cite{M2}. In \cite{M2}, Moree considered  polynomials of the form $f(x) = 2^\alpha d_1 x^2 \pm 2^\alpha d_2 + 1$ where $d_1d_2=\Delta$ and $\alpha\in \mathbb{Z}^+$, where $\left( \frac{\Delta}{p}\right)=-1$ for many consecutive primes $q\geq 3$. The form that we use allows us to vary over $b$ in the family $f(x)=ax^2+bx+c$. 
In \cite{M2} it is reported that Yves Gallot has found the quadratic $f (x) = 64d_1 (x +
728069)^2-64d_2 +1$ of positive discriminant with $d_1= 230849$, $d_1d_2=\Delta = 4472988326827347533$ (taken from \cite[Table 4.3]{JW}) and $g=66715361$ with  $\ell_g(f)=25581$. We also implemented the method given in \cite{M2} and 
using $d_1=373$, $\Delta = 2430946649400343037$ (taken from \cite[Table 4.6]{JW}) we found  $f (x) = 32d_1 (x + 4685199)^2 - 32d_2 + 1$ of positive discriminant and $g=675$
with $\ell_g(f)= 26187$. 
By using the same method we also found 
$f(x)=x^2+3543608x+13598861653501886604$ of negative discriminant and $g=69870828$ with  $\ell_g(f)=35521$. 

\begin{table}
\centering


\begin{tabular}{|l|l|l|l|}
\hline
$f(x)$ &  $g$ &   $\ell_g(f)$&  $\delta(f)$ \\ \hline
\footnotesize $x^2+108656x+2038991585917703148$ & \footnotesize $29823674796$ & \footnotesize $37951$ & \footnotesize $0.999553$\\ \hline
\footnotesize $x^2+609932x+2038991675970432720$ & \footnotesize $70882491394$ & \footnotesize $36041$ & \footnotesize $0.999553$ \\ \hline
\footnotesize $x^2+172676x+2038991590420421808$ & \footnotesize $122605515633473715037016$ & \footnotesize $31801$ & \footnotesize $0.999553$\\ \hline
\end{tabular}
\hskip2cm
\caption{{\tiny Quadratics of Negative Discriminant With Long Artin Prime Production Lengths}}
\label{QuadTable}
\end{table}

\section{The Cubic Case}
\label{The Cubic Case}

In this section we study Artin prime producing polynomials of the form $f(x)=ax^3+b$.
The following proposition plays an important role in our investigations. 
\begin{proposition}
\label{Prop6.4.14}
Let $q$ be a prime number not dividing the integer $m$. The decomposition of $x^3 - m$ modulo $q$ is as follows.

\begin{enumerate}
 \item If $q\equiv 2 \imod{3}$, then $x^3 - m= (x-u)(x^2 +ux + w)$ in $\bF_q[x]$.
 \item If $q \equiv 1 \imod{3}$ and $m^{(q-1)/3} \equiv 1 \imod{q}$, then $x^3 - m= (x-u_1)(x-u_2)(x-u_3)$ in $\bF_q[x]$, where $u_1$, $u_2$, and $u_3$ are distinct elements of $\bF_q$.
 \item If $q \equiv 1 \imod{3}$ and $m^{(q-1)/3} \not\equiv 1 \imod{q}$, then $x^3 - m$ is irreducible in $\bF_q[x]$.
 \item If $q = 3$, then $x^3 - m = (x-a)^3 $ in $\bF_q[x]$.
\end{enumerate}
\end{proposition}

\begin{proof}
See \cite[Proposition 6.4.14]{C}.
\end{proof}

\subsection{Proportion of Primitive Roots}
\label{Proportion of Primitive Roots}

We evaluate $\delta(f)$ for $f(x) = ax^3 + b$ with $(a, b)=1$ by computing the number of solutions to the following two congruence equations using Proposition \ref{Prop6.4.14}. 

\begin{align*}
\#\{s \imod{q}  |~ f(s) \stackrel{q}{\equiv} 1\} &= \left\{\begin{array}{lll}
     1 & \hspace{-5pt} \text{if $q\nmid a, q\mid (b-1) $} \quad &(a)\\
     1 & \hspace{-5pt} \text{if $q\nmid a, q =3$} \quad &(b)\\
     1 & \hspace{-5pt} \text{if $q\nmid a, q \equiv 2 \imod{3}$} \quad &(c)\\
     3 & \hspace{-5pt} \text{if $q\nmid a, q \equiv 1 \imod{3}, (a^2(b-1))^\frac{q-1}{3} \equiv 1 \imod{q}$} \quad &(d)\\
     0 & \hspace{-5pt} \text{if $q\nmid a, q \equiv 1 \imod{3}, (a^2(b-1))^\frac{q-1}{3} \not\equiv 1 \imod{q}$} \quad &(e)\\
     q & \hspace{-5pt} \text{if $q\mid a, q \mid (b-1)$} \quad &(f)\\
     0 & \hspace{-5pt} \text{if $q\mid a, q \nmid (b-1)$} \quad &(g)
\end{array}
\right.
\end{align*}
\begin{align*}
\#\{s \imod{q}  |~ f(s) \stackrel{q}{\equiv} 0\} &= \left\{\begin{array}{lll}
     1 & \hspace{-5pt} \text{if $q\nmid a, q\mid b $} \quad &(A)\\
     1 & \hspace{-5pt} \text{if $q\nmid a, q =3$} \quad &(B)\\
     1 & \hspace{-5pt} \text{if $q\nmid a, q \equiv 2 \imod{3}$} \quad &(C)\\
     3 & \hspace{-5pt} \text{if $q\nmid a, q \equiv 1 \imod{3}, (a^2b)^\frac{q-1}{3} \equiv 1 \imod{q}$} \quad &(D)\\
     0 & \hspace{-5pt} \text{if $q\nmid a, q \equiv 1 \imod{3}, (a^2b)^\frac{q-1}{3} \not\equiv 1 \imod{q}$} \quad &(E)\\
     0 & \hspace{-5pt} \text{if $q\mid a$} \quad &(F)
\end{array}
\right.
\end{align*}
Then the cubic version of \eqref{deltaf} is

\begin{eqnarray}
\label{cubicdf}
\delta(f) & = & \prod\limits_{{\stackrel{\stackrel{(a),((B) \text{ or } (C))}{(b),((A) \text{ or } (B))}}{(c),((A) \text{ or } (C))}}}\left(1 - \frac{1}{q(q-1)}\right) \prod\limits_{{(a),(D)}}\left(1 - \frac{1}{q(q-3)}\right) \prod\limits_{{(a),(E)}}\left(1 - \frac{1}{q^2}\right)\nonumber\\
&  &  \prod\limits_{{(d),(A)}}\left(1 - \frac{3}{q(q-1)}\right)  \prod\limits_{{(d),(D)}}\left(1 - \frac{3}{q(q-3)}\right) \prod\limits_{{(d),(E)}}\left(1 - \frac{3}{q^2}\right) \prod\limits_{{(f),(F)}}\left(1 - \frac{1}{q}\right), 
\end{eqnarray}
where $q$ ranges over the odd primes. The letters in the subscripts of the products refer to the conditions in the number of solutions of congruences modulo $q$. For example $(a),((B) \text{ or } (C))$ indicates that $q$ either satisfies the conditions (a) and (B) or it satisfies the conditions (a) and (C).

Observe that for $f(x)=ax^3+b$  we have
$$
C(f)=\frac{1}{3}\prod\limits_{{(E) \text{ or } (F)}}\left(1 + \frac{1}{q-1}\right) 
\prod\limits_{{(D)}}\left(1 - \frac{2}{q-1}\right),$$
where $q$ ranges over the odd primes. For $f_1(x)=ax^3+(b-1)$ note that if $C(f_1)$ is large then we expect that $(a^2(b-1))^\frac{q-1}{3} \not\equiv 1 \imod{q}$ for many small consecutive primes $q$ where $q\nmid a$ and $q\equiv 1 \imod{3}$. Considering this fact in \eqref{cubicdf} shows that if $(a, b-1)$ does not have small odd prime divisors, then $\delta(f)$ will be more likely to be large. So the problem of finding $f(x)=ax^3+b$ with large $\delta(f)$ is related to finding integers $a$ and $b$ such that $f_1(x)=ax^3+(b-1)$ has large $C(f_1)$. 
\subsection{A formula for $a_{D_1}(f)$}
In order to find an expression for $\tau_D^{-}(f)$ for cubic polynomials
we need to compute some special character sums.
For prime $q\equiv 1\imod{k}$ the Jacobsthal sum $\phi_{q,k}(E)$ and its associated sum $\psi_{q,k}(E)$ 
are defined as 
\begin{align*}
\phi_{q,k}(E) = \sum\limits_{u=1}^{q-1}\left(\frac{u}{q}\right)\left(\frac{u^k + E}{q}\right), ~~\textrm{and}~~
\psi_{q,k}(E) = \sum\limits_{u=1}^{q-1}\left(\frac{u^k + E}{q}\right).
\end{align*}
For odd $k$ it is known that 
$$\psi_{q,k}(E) = \left(\frac{E}{q}\right)\phi_{q,k}(\overline{E})$$
where $\overline{E}$ is the modular inverse of $E$ mod $q$ (see \cite[page 104, equation (5)]{L2}). For $k=3$ this latter identity in combination with the definition of $\psi_{q, k}(E)$ implies
\begin{equation}
\label{CubicLemma1}
\sum\limits_{m=0}^{q-1} \left(\frac{m^3 + E}{q}\right) = \left(\frac{E}{q}\right)(1+\phi_{q,3}(\bar{E})).
\end{equation}
Note that $\phi_{q,3}(\overline{E})$ can be explicitly evaluated. We will describe its computation in the next section. 
\begin{lemma}
\label{numerator}
Let $q$ be an odd prime. Then
$$\sum\limits_{m=0}^{q-1} \left(\frac{am^3+b}{q}\right) = \begin{cases}
q\left(\frac{b}{q}\right) & \text{if $q \mid a$}, \\
0 & \text{if $q \nmid a$ and $($ $q\mid b$ or $q = 3$ or $q \equiv 2 \imod{3}$ $)$}, \\
\left(\frac{b}{q}\right)(1 + \phi_{q,3}(\bar{a}^2\bar{b})) & \text{if $q \nmid a$, $q\nmid b$, $q \equiv 1 \imod{3}$}.
           \end{cases}
$$
\end{lemma}
\begin{proof}
If $q \mid a$ then
$$\sum_{m=0}^{q-1} \left(\frac{am^3 + b}{q}\right) = \sum_{m=0}^{q-1}\left(\frac{b}{q}\right) = q\left(\frac{b}{q}\right).$$
If $q=3$ or $q \equiv 2 \imod{3}$ then the map $x \rightarrow x^3$ from $\mathbb{F}_q \rightarrow \mathbb{F}_q$ is one to one (see \cite[Theorem 4.13]{Lev}). 
So if $q\nmid a$ we have 
$$\left(\frac{a^2}{q}\right)\sum\limits_{m=0}^{q-1} \left(\frac{am^3+b}{q}\right)  = \sum\limits_{k=0}^{q-1}\left(\frac{k^3 + a^2b}{q}\right)=0.$$
Finally assume that $q \nmid a$ and $q \equiv 1 \imod{3}$.
Then from \eqref{CubicLemma1} we have 
$$\left(\frac{a^2}{q}\right)\sum\limits_{m=0}^{q-1} \left(\frac{am^3+b}{q}\right) = \sum\limits_{k=0}^{q-1}\left(\frac{k^3 + a^2b}{q}\right) = \left(\frac{b}{q}\right)(1+\phi_{q,3}(\bar{a}^2\bar{b})).$$\end{proof}
Recall from Theorem \ref{Theorem 1} that for odd square-free $d$ the multiplicative function $a_d(f)$ is defined by
\begin{equation*}
a_d(f) = \frac{\sum_{r\imod{d}} \left(\frac{f(r)}{d} \right)}{\# \{ r\imod{d} | (f(r),d) = 1 \} }.
\end{equation*}
We next find a formula for $a_{q}(f)$ for odd prime $q$.
\begin{lemma}
\label{a_pCubic}
Let $q$ be an odd prime and $f(x)=ax^3+b$ where $(a, b)=1$.
$$a_q(f) = \begin{cases}
\left(\frac{b}{q}\right) & \text{if $q \mid a$}, \\
0 & \text{if $q \nmid a$ and $($ $q \mid b$ or $q = 3$ or $q \equiv 2 \imod{3}$ $)$}, \\
\frac{\left(\frac{b}{q}\right)(1 + \phi_{q,3}(\bar{a}^2\bar{b}))}{q-3} & \text{if $q \equiv 1 \imod{3}$, $a^2b$ is a cubic residue mod $q$}, \\
\frac{\left(\frac{b}{q}\right)(1 + \phi_{q,3}(\bar{a}^2\bar{b}))}{q} & \text{if $q \equiv 1 \imod{3}$, $a^2b$ is a cubic non-residue mod $q$}.
           \end{cases}
$$
\end{lemma}
\begin{proof} The result follows from a straightforward application of Lemma \ref{numerator} and Proposition \ref{Prop6.4.14}. 
Note that when $q \equiv 1 \imod{3}$, $E$ is a cubic residue mod $q$ if and only if $E^\frac{q-1}{3} \equiv 1 \imod{q}$ (see \cite[Theorem 4.13]{Lev}).  Also since $q$ is a prime in the form $3k+1$ then $k=(q-1)/3$ is even, so $-a^2b$ is a cubic residue mod $q$ if and only if $a^2b$ is a cubic residue. 
\end{proof}
The following proposition is a simple consequence of Lemma \ref{a_pCubic} and the multiplicativity  of $a_d(f)$ for odd values of $d$.
\begin{proposition}
\label{a_d1Cubic}
Let $D_1$ be an odd square-free integer.
If $D_1$ has a prime divisor $q$ such that $q\nmid a$ and one of the conditions $q\mid b$,
$q=3$, or $q\equiv 2 \imod{3}$ holds, then $a_{D_1}(f)=0$. 
Otherwise
$$a_{D_1}(f)= \left(\frac{b}{(D_1,a)}\right)\prod\limits_{\stackrel{q \mid D_1, q\nmid a}{(1)}}\frac{\left(\frac{b}{q}\right)(1 + \phi_{q,3}(\bar{a}^2\bar{b}))}{q-3}\prod\limits_{\stackrel{q \mid D_1, q\nmid a}{(2)}}\frac{\left(\frac{b}{q}\right)(1 + \phi_{q,3}(\bar{a}^2\bar{b}))}{q},$$
%
where (1) is the condition that $q \equiv 1 \imod{3}$ and ${a}^2{b}$ is a cubic residue mod $q$ and (2) is the condition that $q \equiv 1 \imod{3}$ and ${a}^2{b}$ is a cubic non-residue mod $q$.
\end{proposition}
\subsection{Computing the Jacobsthal sum $\phi_{q,3}(E)$}
The formula for $a_{D_1}(f)$ given in Proposition \ref{a_d1Cubic} will be useful only if we can compute
the Jacobsthal sum $\phi_{q,3}(E)$ for  $q \equiv 1 \imod{3}$. In this section we obtain formulas to compute $\phi_{q,3}(E)$.
If $q\equiv 1 \imod{3}$ then there are integers $A$ and $B$ uniquely defined by 
$$q=A^2+3B^2,~~~A\equiv -1 \imod{3},~~~B>0.$$
(See \cite[Theorems 3.0.1 and 3.1.1]{BEW} for a proof.)
The following proposition provides convenient formulas for $\phi_{q,3}(E)$ in terms of the representation $q=A^2+3B^2$.

\begin{proposition}
\label{AB}
Let $E$ be an integer not divisible by prime $q\equiv 1 \imod{3}$, then 
$$\phi_{q,3}(E)=\left\{ \begin{array}{ll}
                    -1+2A&{\text {if}}~E^{(q-1)/3}\equiv 1\imod{q},\\
                     -1-A-3B & {\text {if}}~E^{(q-1)/3}\equiv (A-B)/2B\imod{q},\\
                     -1-A+3B & {\text {if}}~E^{(q-1)/3}\equiv (-A-B)/2B\imod{q}.\\
                   \end{array}
 \right.$$
\end{proposition}
\begin{proof}
See \cite[Theorem 6.2.10]{BEW}.
\end{proof}
Formulas given in Propositions \ref{AB} can be used in the implementation of our upcoming cubic procedure.
\subsection{Condition for $\tau_D^{-}(f)=1$}
We employ Propositions \ref{AB} to find a condition on the possible values of $D$ with $\tau_D^{-}(f) = 1$.
\begin{proposition} 
\label{cubic condition}
Assume that $f(x)=ax^3+b$ produces infinitely many primes and that the primes produced by $f(x)$ are uniformly distributed among allowable congruence classes. Then
$\tau_D^{-}(f) = 1$ implies $D \mid 56a$.
\end{proposition}
\begin{proof}
From the definition of $\alpha_j$ in Theorem \ref{Theorem 1}, we conclude that $(\alpha_3 + \alpha_7 - \alpha_1 - \alpha_5)$, $(\alpha_3 + \alpha_5 - \alpha_1 - \alpha_7)$, and $(\alpha_5 + \alpha_7 - \alpha_1 - \alpha_3)$  are at most $1$ and at least $-1$. Thus from  \eqref{tau} we deduce that if $\tau_D^{-}(f) = 1$, then  $a_{D_1}(f) = \pm 1$, where $D_1>1$ is the largest odd square free divisor of $D$.

Let $q$ be a divisor of $D_1$ such that $q\nmid a$. From Lemma \ref{a_pCubic} we know that the
only possible non-zero value of $a_q(f)$ are
$$\text{either}~~\frac{(\frac{b}{q}) (1+\phi_{q,3}(\bar{a}^2 \bar{b}))}{q-3},~~\text{or}~\frac{(\frac{b}{q}) (1+\phi_{q,3}(\bar{a}^2 \bar{b}))}{q}.$$
In the former case $q\equiv 1 \imod{3}$ and $a^2b$ is a cubic residue mod $q$ and
in the latter case $q\equiv 1 \imod{3}$ and $a^2b$ is a cubic non-residue mod $q$.
From Proposition \ref{AB} we know that if $a^2 b$ is a cubic residue mod $q$ then
$(1 + \phi_{q,3}(\bar{a}^2\bar{b}))$ is equal to $2A$. So in this case $a_q(f)=\pm 1$ implies 
$$\frac{\left| \left(\frac{b}{q}\right)(1 + \phi_{q,3}(\bar{a}^2\bar{b}))\right|}{q-3}=\frac{\left|\left(\frac{b}{q}\right)(2A)\right|}{q-3}=1.$$
Since $|A| \leq \sqrt{q}$ (recall
that $q=A^2+3B^2$), from the above identity we conclude that $2\sqrt{q} \geq q-3$ and so $4q\geq (q-3)^2$.
Because $q \equiv 1 \imod{3}$, this is only true if $q = 7$. 

Next if $a^2 b$ is a cubic non-residue modulo $q$, then from Lemma \ref{a_pCubic} we conclude that if $a_q(f)=\pm 1$ then 
\begin{equation}
\label{term}
\frac{\left|\left(\frac{b}{q}\right)(1 + \phi_{q,3}(\bar{a}^2\bar{b}))\right|}{q} = 1,
\end{equation}
which implies $|1 + \phi_{q,3}(\bar{a}^2\bar{b})| = q. $
From Proposition \ref{AB} we know that  $1 + \phi_{q,3}(\bar{a}^2\bar{b})$ is equal to $-A\pm 3B$. 
%
%
So $|1 + \phi_{q,3}(\bar{a}^2\bar{b})| = |-A\pm3B|=q$. Since $|A| \leq \sqrt{q}$ and $|B| \leq (1/\sqrt{3}) \sqrt{q}$ (recall
that $q=A^2+3 B^2$) we have $q=|-A\pm3B|\leq (1+\sqrt{3})\sqrt{q}.$
Because $q \equiv 1 \imod{3}$, this is only true if $q = 7$. 

In summary if $q\nmid a$ and $a_q(f)=\pm 1$, then $q=7$. This shows that if 
$a_{D_1}(f) = \pm 1$, then $D_1 \mid 7a$. 
Finally since $D$ is a fundamental discriminant, $8$ is the greatest power of $2$ that divides $D$. This implies that if $\tau_D^{-}(f) = 1$,  then $D \mid 56a$.
\end{proof}

The above result gives us a convenient way to find a $D$ such that $\tau_D^{-}(f) = 1$.
\subsection{The Cubic Procedure}
We are ready to present our algorithm for finding prime producing cubic polynomials $f(x)=ax^3+b$ and integers $g$ with large $\delta(f)$ and $\ell_g(f)$.
\medskip\par
\noindent{\bf Cubic Procedure}
\begin{enumerate}
\item Select coprime integers $A>0$ and $B$ such that:

(i) The smallest prime factor of $B$ is large.

(ii) $3$ and many consecutive primes $q\equiv 2 \imod{3}$ divide $A$.

(iii) For many consecutive primes $q\equiv 1 \imod{3}$ we have $(A^2B)^\frac{q-1}{3} \not\equiv 1 \imod{3}$.

 \item Set $a=2^\alpha A$, $b=2^\alpha B+1$, and choose $\alpha$ such that $(a,b) = 1$ and $a^2b$ is not a perfect cube. Then form $f(x)=ax^3+b$.
 \item Search through divisors $D$ of $56a$ that are fundamental discriminants and find ones satisfying $\tau_D^{-}(f)=1$. Then select $g$ such that  $\mathbb{Q}(\sqrt{g})$ has fundamental discriminant $D$.
 \item Compute $\ell_g(f)$.
\end{enumerate}

We briefly explain why this works. Recall that our aim is to make $\delta(f)$ in \eqref{cubicdf} as close as possible to $1$. In order to do this we need to ensure that $\#\{s \imod{q} | f(s) \equiv 1 \imod{q}\}$ is zero for as many small primes $q$ as possible. 
For $q=3$ and $q \equiv 2 \imod{3}$,  the equation $f(n) \equiv 1 \imod{q}$ has no solutions only if $q \mid a$ and $q \nmid b-1$. Now because of our choice of $A$, we have that many such small primes (i.e. $3$ and odd primes $q\equiv  2 \imod{3}$) divide $A$. Since $(a, b-1) = 2^\alpha(A,B) = 2^\alpha$ we conclude that such $q$ does not divide $b-1$. Thus  $\#\{s \imod{q} | f(s) \equiv 1 \imod{q}\}$ is zero for such small prime $q$ as we required.
For $q \equiv 1 \imod{3}$, we have
$$(a^2(b-1))^\frac{q-1}{3}=(2^{q-1})^\alpha (A^2B)^\frac{q-1}{3}
 \equiv (A^2B)^\frac{q-1}{3} \imod{q}.$$
Since $A$ and $B$ are such that $(A^2B)^\frac{q-1}{3} \not \equiv 1 \imod{q}$ for many small $q \equiv 1 \imod{3}$, expression in \eqref{cubicdf} shows that such a $q$ does not reduce the value of $\delta(f)$. So we expect a large value for $\delta(f)$. So for this $f$ and $g$ found in step (3) of the procedure we expect to obtain a large $\ell_g(f)$.

We implemented this procedure and found many examples of cubics $f(x)$ and integers $g$ with large $\delta(f)$ and $\ell_g(f)$. A sample of our findings is given in Table \ref{CubicTable}.
Note that these polynomials are all in the form $a(x+d)^3+b$ for integers $b$ and $d$ and $${a=2^3\times 3\times \prod_{\substack{{5\leq q \leq 113}\\{q\equiv 2 \imod{3}}}} q}.$$ 
\begin{table}[ht]
\centering


\begin{tabular}{|l|l|l|l|}
\hline
$f(x)$ & $g$ &   $\ell_g(f)$&  $\delta(f)$ \\ \hline

\scriptsize $16735790906636782452200520x^3 \hspace{-1pt} + \hspace{-1pt} 41975422096126566714360524823960x^2+\hspace{-1pt}$ & & & \\
\scriptsize $35093173864667750962440687534348342360x+$ & \scriptsize $11045$ &	\scriptsize $10011$ & \scriptsize $0.999103$\\
\scriptsize $9779777390330230394129958282301374167637377$ & & &\\ \hline

\scriptsize $16735790906636782452200520x^3 \hspace{-1pt} + \hspace{-1pt} 35691015460148108446082064160320x^2+\hspace{-1pt}$ & & &\\
\scriptsize $25371743542186406147283249113774999040x+$ & \scriptsize $3380$ & \scriptsize $9938$ & \scriptsize $0.999103$\\
\scriptsize $6012020691773711636910512621335820375159417$ & & &\\ \hline

\scriptsize $16735790906636782452200520x^3 \hspace{-1pt} + \hspace{-1pt} 13889869662963197596203821574000x^2+\hspace{-1pt}$ & & &\\
\scriptsize $3842632442258768614989787238447100000x+$ & \scriptsize $45$ & \scriptsize $9472$ & \scriptsize $0.999103$\\
\scriptsize $354354755050296112445641546505463407638457$ & & &\\ \hline	

\scriptsize $16735790906636782452200520x^3 \hspace{-1pt} + \hspace{-1pt} 8188671868499217922819644631320x^2+\hspace{-1pt}$ & & &\\
\scriptsize $1335547815736616945558115580434398040x+$ & \scriptsize $1445$ & \scriptsize $8499$ & \scriptsize $0.999103$\\
\scriptsize $72607947367731671323230658940703008348417$ & & &\\ \hline

\scriptsize $16735790906636782452200520x^3 \hspace{-1pt} + \hspace{-1pt} 39188360629071623422248215626800x^2+\hspace{-1pt}$ & & & \\
\scriptsize $30587691121809274229767399743186204000x+$ & \scriptsize $10125$ & \scriptsize $8243$ & \scriptsize $0.999103$\\
\scriptsize $7958203517101930938186782840516375938678457$ & & &\\ \hline
\end{tabular}
\hskip2cm
\caption{Cubic Polynomials With Long Artin Prime Production Lengths}
\label{CubicTable}
\end{table}

\section{Concluding Remarks}
\label{Concluding Remarks}

In this paper we focused on the problem of maximizing $\delta(f)$, as defined in \eqref{deltaf} for any prime producing polynomial $f(x)$, when $f(x)$ varies over prime producing polynomials of fixed degrees. We note that  $\delta(f)$ is well defined for any polynomial $f(x)$ with the property that $N_q(f)\neq q$ for any primes $q\geq3$. From the above investigations we speculate that
$$\sup_{f,{\rm deg}(f)=n} \delta(f) =1,$$
for $n=1, 2$, or $3$. More generally one can ask the following question:
\begin{question}
\label{question 1}
Is it true that $\sup_{f,{\rm deg}(f)=n} \delta(f) =1$ and $\inf_{f,{\rm deg}(f)=n} \delta(f)=0$?
\end{question}
It turns out that the answer to the infimum question for the linear, quadratic, and cubic polynomials is simple. 
We need only to consider the term
$$\prod\limits_{\stackrel{q> 2}{q \mid (a,b-1)}}\left(1-\frac{1}{q}\right)$$
that is present in equations \eqref{lineardelta}, \eqref{cubicdf}, and 
in \eqref{Quadratic Delta 2} for $f(x)=ax^2+b$.
Since $\prod_{q>2} (1-1/q)=0$ by defining integer $a$ as a product of consecutive primes starting from $3$  and setting $b = a+1$, we find $f(x)=ax^n+b$ with $\delta(f)$ arbitrarily close to $0$. So
$\inf_{f,{\rm deg}(f)=n} \delta(f)=0,$
for $n=1, 2,$ or $3$. 

The answer to the supremum question for polynomials of degree $n$ is positive. In order to see this, for $y>0$ let 
$q_1, q_2, \ldots, q_m$ be all the odd primes not exceeding $y$. Let $a=q_1q_2 \ldots q_m$. Take an integer $b$ such that $b\equiv 2\imod{a}$. Note that $(a, b)=(a, b-1)=1$. Form $f_y(x)=ax^n+b$. From \eqref{deltaf} we have $$\delta(f_y)=\prod_{q>y} \left( 1-\frac{1}{q(q-N_q(f_y))}\right),$$ 
where $N_q(f_y)\neq q$. It is clear that $\delta(f_y)\rightarrow 1$ as $y\rightarrow\infty$. 

We also note that it is possible to construct quadratic polynomials $f_y(x)=x^2+bx+c$ with $b\neq 0$ and $\delta(f_y)$ arbitrarily close to one.  In order to do this let $\Delta$ be an integer with the property that $\left(\frac{\Delta}{q_i}\right)=-1$ for $i=1, \ldots, m$ (the existence of such $\Delta$ is a consequence of the law of quadratic reciprocity, see \cite{LLS} for details). Choose integers  $b^\prime\neq 0$ and $c$ such that
$$(b^\prime)^2-\Delta=c-1,$$
and set $b=2b^\prime$.
Now for $f_y(x)=x^2+bx+c$ from \eqref{Quadratic Delta 2} we have 
\begin{equation*}
\delta(f_y) =  \prod\limits_{{q > y}}\left(1 - \frac{1 + \left( \frac{\Delta}{q}\right)}{q\left(q-1-\left(\frac{\Delta}{q}\right)\right)}\right).
\end{equation*}
It is clear that $\delta(f_y)\rightarrow 1 $ as $y\rightarrow\infty$. 

We can also consider the following question.
\begin{question}
\label{question 2}
Is it true that for any polynomial $f(x)$ we have $0<\delta(f)<1$?
\end{question}
For linear $f(x)$ the answer is yes  by 
\eqref{lineardelta} and the fact that we have $$0<\prod_{q} \left(1-\frac{1}{q(q-1)}\right)=A<1.$$ 
In \cite[Proposition 4]{M2}, it is proved that $\delta(f)<1$ for quadratic polynomials.

The next question is motivated by Question \ref{question 1} and the relation \eqref{expected} between $\delta_g(f)$ and $\ell_g(f)$.
\begin{question}
\label{question 3}
Is it true that $\displaystyle{\sup_{\substack{{g\in \mathbb{Z}}\\{f,{\rm deg}(f)=n}}}\ell_g(f)=\infty}$?
\end{question}
In \cite[Theorem 2]{M2} it is conditionally proved that for quadratic polynomials the answer to the above question is positive. It appears that the proof extends to prime producing polynomials of the form $ax^n+b$.

Another question motivated by the size of the leading coefficients $a(f)$ of polynomials $f(x)$ in our findings in this research is the following.

\noindent {\bf Question 6.4} Is it true that $\displaystyle{\sup_{\substack{{g\in \mathbb{Z}}\\{f, {\rm deg}(f)=n}}} \frac{\ell_g(f)}{a(f)}}=\infty$?

Following the procedure described in this paper one may speculate that for linear polynomials $f(x)=ax+b$  the quantity ${\ell_g(f)}/{a(f)}$ is bounded. Note that for a prime producing polynomial $f_{y}(x)=ax+b$, where $a$ is the product of all the odd primes $\leq y$ and $(a, b-1)=1$, and a suitable $g$ coming from $\tau_D^-(f_y)=1$, the ratio of the expected value of $\ell_g(f_y)$ by $a(f_y)$ can be estimated as
$$\frac{\delta(f_y)}{a(f_y)(1-\delta(f_y))}=\frac{\prod_{q>y} \left( 1-\frac{1}{q(q-1)}\right)}{\left( \prod_{3\leq q \leq y} q \right) \left( 1-  \prod_{q>y}\left( 1-\frac{1}{q(q-1)}\right)\right)}\approx \frac{1-1/y\log{y}}{e^y (1/y\log{y})}.$$
It is clear that the latter expression approaches zero as $y\rightarrow \infty$.
So motivated by this observation we may speculate that the answer to the above question for linear polynomials is negative. In contrast, our investigations for the quadratic case leave open the possibility of the  existence of  sequences of quadratics $f_n$ and integers $g_n$ with the property that $\ell_{g_n}(f_n)/a(f_n)\rightarrow \infty$ as $n\rightarrow \infty$.




Finally, problems similar to the one discussed in \cite{M2} and this paper can be considered for primes generated by a family of polynomials. For example for two polynomials $f_1(x)$ and $f_2(x)$ and a fixed integer $g$, we can consider integers $n$ where both $f_1(n)$ and $f_2(n)$ are prime. Then the Artin prime production length $\ell_g(f_1, f_2)$ is the number of such $n$ in a row where both primes have $g$ as a primitive root. One can develop procedures, in line with the one developed in this paper for the case of a single polynomial, for finding integers $g$ and polynomials $f_1(x)$ and $f_2(x)$ with large $\ell_g(f_1, f_2)$. We have done some preliminary experiments for the case of two quadratic polynomials.
We present a sample of our results in Table \ref{MultipleTable}.

\begin{table}[ht]
\centering

\begin{tabular}{|l|l|l|}
\hline
$f_1(x)$, $f_2(x)$ & $g$ & $\ell_g(f_1, f_2)$ \\ \hline

\footnotesize $x^2+77851376x+9829839069358873548$ & & \\
\footnotesize $x^2+77851376x+5695745484831292308$ & \footnotesize $7203$ & \footnotesize $11966$ \\ \hline

\footnotesize $x^2+24444296x+9828473241074334108$  & & \\
\footnotesize $10597x^2+259036204712x+1583526759288000168$ & \footnotesize $108$ & \footnotesize $10724$ \\ \hline

\footnotesize $x^2+65043728x+13599916185850506684$  & & \\
\footnotesize $x^2+65043728x+6850377136300469580$ & \footnotesize $21675$ & \footnotesize $10043$ \\ \hline

\footnotesize $x^2+64233308x+13599889993676627904$  & & \\
\footnotesize $x^2+64233308x+6850350944126590800$ & \footnotesize $48$ & \footnotesize $9340$ \\ \hline

\footnotesize $x^2+4206728x+13598862938352588684$  & & \\
\footnotesize $x^2+4206728x+6849323888802551580$ & \footnotesize $3468$ & \footnotesize $9247$ \\ \hline
\end{tabular}
\hskip2cm
\caption{Pair of Quadratics With Long Artin Prime Production Lengths}
\label{MultipleTable}
\end{table}

\newpage
\noindent {\bf Acknowledgements}.   The authors thank Pieter Moree for his suggestions and encouragement on this work and for his detailed comments on an earlier draft of this paper. We are also grateful to Adam Felix and Michael Jacobson for their comments on an earlier draft of this paper. We also thank the referee for many helpful comments and suggestions.

\end{document}